\documentclass[12pt]{amsart}
\usepackage{a4wide}
\usepackage{amsmath}
\usepackage[utf8]{inputenc}
\usepackage{amssymb}
\usepackage{amsopn}
\usepackage{epsfig}
\usepackage{amsfonts}
\usepackage{latexsym}
\usepackage{graphicx}
\usepackage{enumitem}
\usepackage{color}
\newtheorem{theorem}{Theorem}[section]
\newtheorem{lemma}[theorem]{Lemma}

\newtheorem{corollary}[theorem]{Corollary}

\theoremstyle{definition}

\theoremstyle{remark}

\numberwithin{equation}{section}

\newcommand{\B}{\ensuremath{\mathcal{B}}}
\renewcommand{\c}{ {\mathbf{c}}}
\renewcommand{\t}{ {\mathbf{t}}}

\newcommand{\U}{{  {\mathcal{U}}}}

\newcommand{\set}[1]{\left\{#1\right\}}
\newcommand{\la}{\lambda}

\newcommand{\f}{\infty}
\renewcommand{\d}{\mathbf{d}}
\newcommand{\om}{\omega}
\newcommand{\al}{\alpha}

\newcommand{\si}{\sigma}
\renewcommand{\th}{\theta}
\newcommand{\ra}{\rightarrow}
\newcommand{\de}{\delta}

\begin{document}

\title{univoque bases and Hausdorff dimensions}
\author{Derong Kong}
\address{School of Mathematical Science, Yangzhou University, Yangzhou, JiangSu 225002, People's Republic of China} 

\curraddr{Mathematical Institute, University of Leiden, PO Box 9512, 2300 RA Leiden, The Netherlands} 
\email{derongkong@126.com}

\author{Wenxia Li}
\address{Department of Mathematics, Shanghai Key Laboratory of PMMP, East China Normal University, Shanghai 200062,
People's Republic of China}
\email{wxli@math.ecnu.edu.cn}

\author{Fan L\"{u} }
\address{Department of Mathematics, Sichuan Normal University, Chengdu 610068, People's Republic of China}
\email{lvfan1123@163.com}

\author{Martijn de Vries}
\address{Tussen de Grachten 213, 1381 DZ, Weesp, The Netherlands}
\email{martijndevries0@gmail.com}

\date{\today}
\dedicatory{}

%\begin{frontmatter}

\begin{abstract}
Given a positive integer $M$ and a real number $q >1$, a \emph{$q$-expansion} of a real number $x$ is a sequence $(c_i)=c_1c_2\cdots$ with $(c_i) \in \{0,\ldots,M\}^\infty$ such that
\[x=\sum_{i=1}^{\infty} c_iq^{-i}.\]
 It is well known that if $q \in (1,M+1]$, then each $x \in I_q:=\left[0,M/(q-1)\right]$ has a $q$-expansion. Let  $\mathcal{U}=\mathcal{U}(M)$ be the set of \emph{univoque bases} $q>1$ for which $1$ has a unique $q$-expansion. 
 The main object of this paper is to provide new characterizations of $\mathcal{U}$ and  to show that the Hausdorff dimension of the set of numbers $x \in I_q$ with a unique $q$-expansion changes the most if $q$ ``crosses'' a univoque base.

Denote by $\mathcal{B}_2=\B_2(M)$ the set  of  $q \in (1,M+1]$ such that   there exist numbers   having precisely two distinct $q$-expansions.
As a by-product of our results,  we obtain an answer to a question of Sidorov (2009) and prove that 
\[\dim_H(\mathcal{B}_2\cap(q',q'+\delta))>0\quad\textrm{for any}\quad \delta>0,\]
where $q'=q'(M)$ is the Komornik-Loreti constant. 
\end{abstract}

\keywords{univoque bases \and univoque sets\and   Hausdorff dimensions\and  generalized Thue-Morse sequences}

\subjclass{11A63\and  37B10\and 28A78}

\maketitle

\section{Introduction}\label{sec: Introduction}

Non-integer base expansions  have received much attention since the pioneering works of R\'{e}nyi \cite{Renyi_1957} and Parry \cite{Parry_1960}.
Given a positive integer $M$ and a real number $q \in (1,M+1]$, a sequence $(d_i)=d_1d_2\cdots$  with \emph{digits} $d_i\in\set{0,1,\ldots,M}$ is called a \emph{$q$-expansion of $x$}  or an \emph{expansion of $x$ in base $q$} if
\[
x=\pi_q((d_i)):=\sum_{i=1}^\f\frac{d_i}{q^i}.
\]

It is well known that each $x \in I_q:=[0,M/(q-1)]$ has a $q$-expansion. One such expansion - the \emph{greedy $q$-expansion} - can be obtained by performing the so called \emph{greedy algorithm} of R\'{e}nyi which is defined recursively as follows: if $d_1, \ldots,d_{n-1}$ is already defined (no condition if $n=1$), then $d_n$ is the largest element of $\set{0,\ldots,M}$ satisfying $\sum_{i=1}^n d_i q^{-i} \le x$. Equivalently, $(d_i)$ is the greedy $q$-expansion of $\sum_{i=1}^\f d_i q^{-i}$ if and only if $\sum_{i=n+1}^\f d_i q^{-i+n}<1$  whenever $d_n < M$, $n=1,2,\ldots$.
Hence if $1<q < r \le M+1$, then the greedy $q$-expansion of a number $x \in I_q$  is also the greedy expansion in base $r$ of a number in $I_r$.

Let $\U_q$ be the \emph{univoque set} consisting of numbers  $x\in I_q$ such that $x$ has a unique $q$-expansion, and let $\U_q'$ be the set of corresponding expansions.
Note that a sequence $(c_i)$ belongs to $\U_q'$ if and only if both the sequences $(c_i)$ and $(M-c_i):=(M-c_1)(M-c_2)\cdots$ are greedy $q$-expansions, hence $\U_q' \subseteq \U_r'$ whenever $1< q < r \le M+1$.
Many works are devoted to the univoque sets $\U_q$ (see, e.g., \cite{Erdos_Joo_Komornik_1990,Erdos_Joo_1992,Glendinning_Sidorov_2001}). Recently, de Vries and Komornik  investigated their topological properties in \cite{DeVries_Komornik_2008}.    Komornik  et al. considered their Hausdorff dimension in \cite{Komornik_Kong_Li_2015_1}, and showed  that the dimension function $D: q\mapsto\dim_H\U_q$ behaves like a Devil's staircase  on $(1,M+1]$. For more information on the univoque set $\U_q$ we refer to the survey paper \cite{Komornik_2011} and the references therein.

There is an intimate connection between the set $\U_q$ and the set of \emph{univoque bases} $\U=\U(M)$  consisting of numbers $q>1$ such that $1$ has a unique $q$-expansion over the alphabet $\set{0,1,\ldots,M}$. For instance, it was shown in \cite{DeVries_Komornik_2008} that $\U_q$ is closed if and only if $q$ does not belong to the set $\overline{\U}$. It is well-known  that $\U$ is a Lebesgue null set of full Hausdorff dimension (cf.~\cite{Darczy_Katai_1995,Erdos_Horvath_Joo_1991,Komornik_Kong_Li_2015_1}). Moreover,   the smallest  element of $\U$ is  the \emph{Komornik-Loreti constant} (cf.~\cite{Komornik_Loreti_1998,Komornik_Loreti_2002})
\[q'=q'(M),\] while the largest element of $\U$ is (of course) $M+1$.
Recently, Komornik and Loreti  showed in \cite{Komornik_Loreti_2007}  that its closure $\overline{\U}$ is a \emph{Cantor set} (see also, \cite{DeVries_Komornik_Loreti_2016}), i.e., a nonempty closed set having neither isolated  nor interior points.  Writing the open set  $(1,M+1] \setminus \overline{\U}=(1,M+1)\setminus \overline{\U}$ as the disjoint union of its connected components, i.e.,
\begin{equation}\label{eq:11}
(1,M+1]\setminus\overline{\U}= (1,q')\cup\bigcup(q_0,q_0^*),
\end{equation}
the left endpoints $q_0$ in (\ref{eq:11}) run over the whole set $\overline{\U}\setminus\U$, and the right endpoints $q_0^*$ run through a subset of $\U$ (cf.~\cite{DeVries_Komornik_2008}).
Furthermore,  each left endpoint $q_0$ is algebraic, while each right endpoint $q_0^*\in\U$   is transcendental (cf.~\cite{Kong_Li_2015}).

De Vries  showed in \cite{deVries_2009},  roughly speaking, that the sets $\U_q'$ change the most if we cross a univoque base. More precisely, it was shown that $q \in \U$ if and only if $\U_r' \setminus \U_q'$ is uncountable for each $r \in (q,M+1]$ and $r \in \overline\U$ if and only if $\U_r' \setminus \U_q'$ is uncountable for each $q \in (1,r)$.

The main object of this paper is to provide similar characterizations of $\U$ and $\overline{\U}$ in terms of the Hausdorff dimension of the sets $\U_r' \setminus \U_q'$ after a natural projection. Furthermore, we characterize the sets $\U$ and $\overline{\U}$ by looking at the Hausdorff dimensions of $\U$ and $\overline{\U}$ locally.

\begin{theorem}
\label{th:11}
Let  $q\in(1, M+1]$. The following statements are equivalent.
\begin{enumerate}
\item [{\rm(i)}]
$q\in\U$.

\item[{\rm(ii)}]  $\dim_H\pi_{M+1}(\U_r'\setminus\U_q')>0$  for any $r\in(q, M+1]$.

\item[{\rm(iii)}] $\dim_H\U\cap(q,r)>0$ for any $r\in(q, M+1]$.

\end{enumerate}
\end{theorem}

\begin{theorem}\label{th:12}
Let  $q\in(1, M+1]$. The following statements are equivalent.
\begin{enumerate}
\item[{\rm(i)}]
$q\in\overline{\U}\setminus(\bigcup\set{q_0^*}\cup\set{q'})$.

\item[{\rm(ii)}]    $\dim_H\pi_{M+1}(\U_q'\setminus\U_p')>0$  for any $p\in(1, q)$.

\item[{\rm(iii)}] $\dim_H\U\cap(p, q)>0$ for any $p\in(1, q)$.

\end{enumerate}
 \end{theorem}
 
 It follows at once from Theorems \ref{th:11} and \ref{th:12} that $\U$ (or, equivalently, $\overline{\U}$) does not contain isolated points. 
 
 We remark that the projection map $\pi_{M+1}$ in Theorem \ref{th:11} (ii) can be replaced by $\pi_\rho$ for any $r\le \rho\le M+1$. Similarly, the projection map $\pi_{M+1}$ in Theorem \ref{th:12} (ii) can also be replaced by $\pi_\rho$ with $q\le \rho\le M+1$. We also point out that Theorems \ref{th:11}   and  \ref{th:12}   strengthen  the main result  of \cite{deVries_2009} where  the cardinality of the sets $\U_q' \setminus \U_p'$ with $1 < p < q \leq M+1$ was determined.

Let $\B_2$ be the set of bases $q \in \left(1,M+1\right]$ for which there exists a number $x\in[0, M/(q-1)]$ having exactly two $q$-expansions. It was asked by Sidorov \cite{Sidorov_2009} whether $\dim_H\B_2\cap(q', q'+\de)>0 $ for any $\de>0$,  where $q'$ is the Komornik-Loreti constant. Since $\U\subseteq\B_2$ (see \cite[Lemma 3.1]{Sidorov_2009}\footnote{This also follows directly from the observation that $q^{-1}$ has exactly two $q$-expansions whenever $q \in \U$.}),  Theorem \ref{th:11} answers this question in the affirmative.
\begin{corollary}\label{cor:13}
$\dim_H\B_2\cap(q', q'+\de)>0$ {for any} $\de>0.$
\end{corollary}

The rest of the paper is arranged as follows. In Section \ref{sec: univoque set} we recall some properties of unique $q$-expansions. The proof of Theorems \ref{th:11} and \ref{th:12} will be given in Section \ref{sec: proofs}.

\section{Preliminaries}\label{sec: univoque set}
In this section we recall some properties of the univoque set $\U_q$. Throughout this paper, a \emph{sequence} $(d_i)=d_1d_2\cdots$ is an element of $\set{0,\ldots,M}^\f$ with each digit $d_i$ belonging to the \emph{alphabet} $\set{0,\ldots,M}$. Moreover, for a \emph{word} $\c=c_1\cdots c_n$ we mean a finite string of digits with each digit  $c_i$ from $\set{0,\ldots,M}$. For two words $\c=c_1\cdots c_n$ and $\d=d_1\cdots d_m$ we denote by $\c\d=c_1\cdots c_nd_1\cdots d_m$ the concatenation of the two words. For an integer $k\ge 1$ we denote by $\c^k$   the $k$-times concatenation of $\c$ with itself, and    by $\c^\f$   the infinite repetition of $\c$. 

For a sequence $(d_i)$ we denote its \emph{reflection} by $\overline{(d_i)}:=(M-d_1)(M-d_2)\cdots$. Accordingly, for a word $\c=c_1\cdots c_n$ we denote its  {reflection} by $\overline{\c}:=(M-c_1)\cdots (M-c_n)$.   If $c_n<M$ we denote by 
$
\c^+:=c_1\cdots c_{n-1}(c_n+1).
$
If $c_n>0$ we write  $\c^-:=c_1\cdots c_{n-1}(c_n-1)$.  

 We will use systematically the lexicographic ordering $<, \le, >$ and $\ge$ between sequences and between words. For two sequences $(c_i), (d_i)\in\set{0,1,\ldots, M}^\infty$ we say that $(c_i)<(d_i)$ if there exists an integer $n\ge 1$ such that $c_1\ldots c_{n-1}=d_1\ldots d_{n-1}$ and $c_n<d_n$. Furthermore, we write $(c_i)\le (d_i)$ if $(c_i)<(d_i)$ or $(c_i)=(d_i)$. Similarly, we say $(c_i)>(d_i)$ if $(d_i)<(c_i)$, and $(c_i)\ge (d_i)$ if $(d_i)\le (c_i)$. We extend this definition to words in the obvious way.  For example, for two words $\c$ and $\d$ we write $\c<\d$ if $\c0^\f<\d0^\f$. 

A sequence is called \emph{finite} if it has a last nonzero element. Otherwise it is called \emph{infinite}. So $0^{\infty}:=00\cdots$ is considered to be infinite.
For $q\in (1,M+1]$ we denote by
\[\al(q)=(\al_i(q))\]
the \emph{quasi-greedy} $q$-expansion of $1$ (cf.~\cite{Daroczy_Katai_1993}), i.e., the lexicographically largest \emph{infinite}  $q$-expansion of $1$.  Let
$\beta(q)=(\beta_i(q))$
 be the \emph{greedy} $q$-expansion of $1$ (cf.~\cite{Parry_1960}), i.e., the lexicographically largest  $q$-expansion of $1$. For convenience, we set $\alpha(1)=0^{\infty}$ and $\beta(1)=10^{\infty}$, even though $\alpha(1)$ is not a $1$-expansion of $1$.

  Moreover, we endow the set $\set{0,\ldots,M}$ with the discrete topology and the set of all possible sequences $\set{0,1,\ldots,M}^{\infty}$ with the Tychonoff product topology.

The following properties of $\alpha(q)$ and $\beta(q)$ were established in  ~\cite{Parry_1960}, see also ~\cite{Baiocchi_Komornik_2007}.
\begin{lemma}\label{lem:21}\mbox{}

\begin{itemize}
\item[{\rm(i)}] The map $q\mapsto \al(q)$ is an increasing bijection from $[1, M+1]$ onto the set of all {\rm infinite} sequences $(\al_i)$ satisfying
\[
\al_{n+1}\al_{n+2}\cdots\le \al_1\al_2\cdots\quad\textrm{whenever}\quad \al_n<M.
\]

\item[{\rm(ii)}] The map $q\mapsto\beta(q)$ is an increasing bijection from $[1, M+1]$ onto the set of all sequences $(\beta_i)$ satisfying
\[
\beta_{n+1}\beta_{n+2}\cdots<\beta_1\beta_2\cdots\quad\textrm{whenever}\quad \beta_n<M.
\]

\end{itemize}
\end{lemma}

\begin{lemma}
  \label{lem:22}\mbox{}
  
  \begin{itemize}
    \item[{\rm(i)}] $\beta(q)$ is infinite if and only if $\beta(q)=\al(q)$.

    \item[{\rm(ii)}] If $\beta(q)=\beta_1\cdots\beta_m 0^\f$ with $\beta_m>0$, then
    $\al(q)=(\beta_1\cdots\beta_m^-)^\f.$

\item[{\rm(iii)}] The map $q\mapsto \al(q)$ is left-continuous, while the map  $q\mapsto \beta(q)$ is right-continuous.
  \end{itemize}
\end{lemma}

In order to investigate the unique expansions we need the following lexicographic characterization of $\U_q'$ (cf.~\cite{Baiocchi_Komornik_2007}).
\begin{lemma}\label{lem:23}
Let $q\in(1,M+1]$. Then $(d_i)\in\U_q'$ if and only if
\[
\left\{
\begin{array}{lll}
d_{n+1}d_{n+2}\cdots <\al_1(q)\al_2(q)\cdots&\textrm{ whenever }&d_n<M,\\
d_{n+1}d_{n+2}\cdots>\overline{\al_1(q)\al_2(q)\cdots}&\textrm{ whenever} &d_n>0.
\end{array}
\right.
\]
\end{lemma}

Note that $q\in\U$ if and only if $\al(q)$ is the unique $q$-expansion of $1$. Then Lemma \ref{lem:23} yields a characterization of $\U$ (see also, \cite{Erdos_Joo_Komornik_1990} and \cite{Komornik_Loreti_2002}).
\begin{lemma}
  \label{lem:24}
 Let $q \in (1,M+1)$. Then $q\in\U$ if and only if $\al(q)=(\al_i(q))$ satisfies
  \[
  \overline{\al(q)}<\al_{n+1}(q)\al_{n+2}(q)\cdots<\al(q)\quad\textrm{for all}\quad n\ge 1.
  \]
\end{lemma}

Consider a connected component $(q_0, q_0^*)$ of $(q', M+1)\setminus\overline{\U}$ as in (\ref{eq:11}). Then there exists a (unique)  word $\t=t_1\cdots t_p$ such that (cf.~\cite{DeVries_Komornik_2008,Kong_Li_2015})
\[
\alpha(q_0)=\t^\f\quad\textrm{and}\quad \al(q_0^*)=\lim_{n\ra\f}g^n(\t),
\]
where $g^n=\underbrace{g\circ\cdots\circ g}_n$ denotes the $n$-fold composition of $g$ with itself, and
 \begin{equation}\label{eq:g}
g(\c) :=\c^+\overline{\c^+}\quad\textrm{for any word}\quad \c=c_1\ldots c_k\textrm{ with }c_k<M.
\end{equation}
We point out that the word $\t=t_1\ldots t_p$ in the definitions of $\alpha(q_0)$ and $\al(q_0^*)$ is called an \emph{admissible block} in \cite[Definition 2.1]{Kong_Li_2015} which satisfies the following lexicographical inequalities: $t_p<M$ and for any $1\le i\le p$ we have 
\[
\overline{t_1\cdots t_p}\le t_i\cdots t_pt_1\cdots t_{i-1}\quad\textrm{and}\quad t_i\cdots t_p\;\overline{t_1\ldots t_{i-1}}\le t_1\ldots t_p^+.
\]
We also mention that the limit $\lim_{n\ra\f}g^n(\t)$ stands for the infinite sequence beginning with
$
\t^+\overline{\t}\,\overline{\t^+}\t^+\;\overline{\t^+}\t\,\t^+\overline{\t}\cdots,
$
and the existence of this limit was shown by Allouche \cite{Allouche_Cosnard_1983}.

 In this case $(q_0, q_0^*)$ is called the  \emph{connected component  generated by $\t$}. The closed interval $[q_0, q_0^*]$ is the so called \emph{admissible interval generated by $\t$} (see \cite[Definition 2.4]{Kong_Li_2015}). Furthermore,  the sequence
 \[
 \al(q_0^*)=\lim_{n\ra\f}g^n(\t)=\t^+\,\overline{\t}\;\overline{\t^+}\,\t^+~ \overline{\t^+}\, \t\, \t^+\,\overline{\t}\cdots
 \]
   is  a  {generalized Thue-Morse sequence} (cf.~\cite[Definition 2.2]{Kong_Li_2015}, see also  \cite{Allouche_Shallit_1999}).

The following lemma for the generalized Thue-Morse sequence  $\al(q_0^*)$ was established in \cite[Lemma 4.2]{Kong_Li_2015}.
\begin{lemma}\label{lem:25}
Let $(q_0, q_0^*)\subset(q', M+1)\setminus\overline{\U}$ be a connected component generated by $t_1\cdots t_p$.  Then the sequence $(\th_i)=\al(q_0^*)$ satisfies
\[
\overline{\th_1\cdots \th_{2^n p-i}}<\th_{i+1}\cdots \th_{2^n p}\le\th_1\cdots \th_{2^n p-i}
\]
for any $n\ge 0$ and any $0\le i<2^n p$.
\end{lemma}

Finally,  we recall some topological properties of $\U$ and $\overline{\U}$ which were essentially established in \cite{DeVries_Komornik_2008,Komornik_Loreti_2007} (see also, \cite{DeVries_Komornik_Loreti_2016}).
\begin{lemma}\label{lem:26}\mbox{}

\begin{itemize}
\item[{\rm(i)}] If $q\in\U$, then there exists a {\rm decreasing} sequence $(r_n)$ of elements in $\bigcup\set{q_0^*}$  that  converges  to $q$ as $n\ra\f$;
\item[{\rm(ii)}] If $q\in\overline{\U}\setminus(\bigcup\set{q_0^*}\cup\set{q'})$, then there exists an {\rm increasing} sequence $(p_n)$ of elements in $\bigcup\set{q_0^*}$ that  converges to $q$ as $n\ra\f$.
\end{itemize}
\end{lemma}

We remark here that the bases $q_0^*$ are called \emph{de Vries-Komornik numbers} which were shown to be transcendental in \cite{Kong_Li_2015}. By Lemma \ref{lem:26} it follows that the set of  de Vries-Komornik numbers is dense in $\overline{\U}$.

\section{Proofs of Theorems \ref{th:11} and \ref{th:12} }\label{sec: proofs}

\subsection{ Proof  of Theorem \ref{th:11} for (i) $\Leftrightarrow$ (ii).}

For each connected component $(q_0, q_0^*)$ of $(q', M+1)\setminus\overline{\U}$  we  construct a sequence of bases $(r_n)$ in $\U$ strictly decreasing to $q_0^*$.
\begin{lemma}\label{lem:31}
Let $(q_0, q_0^*)\subset (q', M+1)\setminus\overline{\U}$ be a connected component generated by  $t_1\cdots t_p$, and let $(\th_i)=\al(q_0^*)$. Then for each $n\ge 1$, the number $r_n\in\U$ determined by 
\[
\al(r_n)=\beta(r_n)=\th_1\cdots \th_{2^n p}(\th_{2^n p+1}\cdots \th_{2^{n+1}p})^\f,
\]
belongs to $\U$.
Furthermore, $(r_n)$ is a strictly decreasing sequence that converges to $q_0^*$.
\end{lemma}
\begin{proof}
Using (\ref{eq:g}) one may verify that the sequence $(\th_i)$ satisfies
\[
\th_{2^n p+k} =\overline{\th_k }\quad\textrm {for all}\quad1\le k<2^n p;\quad\th_{2^{n+1}p} =\overline{\th_{2^n p} }\,^+
\]
for all $n\ge 0$.
Now fix $n\ge 1$. We claim   that
\begin{equation}\label{eq:31}
\si^i\left(\th_1\cdots \th_{2^n p}(\th_{2^n p+1}\cdots \th_{2^{n+1}p})^\f\right)<\th_1\cdots \th_{2^n p}(\th_{2^n p+1}\cdots \th_{2^{n+1}p})^\f
\end{equation}
for all $i\ge 1$, where $\si$ is the left shift on $\set{0,\ldots, M}^\f$ defined by $\si((c_i))=(c_{i+1})$. By periodicity it  suffices  to prove (\ref{eq:31}) for $0<i<2^{n+1}p$. We distinguish between the following three cases: (I) $0< i<2^np$; (II) $i=2^n p$; (III) $2^n p<i<2^{n+1}p$.

Case (I). $0< i<2^n p$. Then by Lemma \ref{lem:25} it follows that
\[
\th_{i+1}\cdots \th_{2^n p}\le \th_1\cdots \th_{2^n p-i}
\]
and
\[
\th_{2^n p+1}\cdots \th_{2^np +i}=\overline{\th_1\cdots\th_i}<\th_{2^n p-i+1}\cdots \th_{2^n p}.
\]
This implies (\ref{eq:31}) for $0<i<2^n p$.

Case (II). $i=2^n p$. Note by \cite{Komornik_Loreti_2002} that $\al_1(q')=[M/2]+1$ (see also, \cite{Baker_2014}), where $[y]$ denotes the integer part of a real number $y$. Then by using  $q_0^*>q'$ in Lemma \ref{lem:21} we have
\[\th_1=\al_1(q_0^*)\ge\al_1(q')>\overline{\al_1(q')}\ge\overline{\th_1}.\]
This, together with $n\ge 1$,  implies
\[
\th_{2^n p+1}\cdots\th_{2^{n+1} p}=\overline{\th_1\cdots\th_{2^np}}\,^+<\th_1\cdots\th_{2^n p}.
\]
So, (\ref{eq:31}) holds true for $i=2^n p$.

Case (III). $2^n p< i<2^{n+1}p$. Write $j=i-2^n p$. Then $0< j<2^n p$. Once again, we infer from Lemma \ref{lem:25} that
\[
\th_{i+1}\cdots \th_{2^{n+1}p}=\overline{\th_{j+1}\cdots \th_{2^n p}}\,^+\le \th_1\cdots\th_{2^n p-j}
\]
and
\[
\th_{2^n p+1}\cdots\th_{2^n p+j}=\overline{\th_1\cdots\th_j}<\th_{2^np-j+1}\cdots\th_{2^n p}.
\]
This yields (\ref{eq:31}) for $2^n p<i<2^{n+1}p$.

Note by Lemma \ref{lem:25} that
\[
\sigma^i\left(\th_1\cdots \th_{2^n p}(\th_{2^n p+1}\cdots \th_{2^{n+1}p})^\f\right)>\overline{\th_1\cdots \th_{2^n p}(\th_{2^n p+1}\cdots \th_{2^{n+1}p})^\f}
\]
for any $i\ge 0$.
 Then by (\ref{eq:31}) and  Lemma \ref{lem:24} it follows that there exists $r_n\in\U$ such that
\[\al(r_n)=\beta(r_n)=\th_1\cdots\th_{2^n p}(\th_{2^n p+1}\cdots \th_{2^{n+1}p})^\f.\]

In the following we prove $r_n\searrow q_0^*$ as $n\ra\f$. For $n\ge 1$ we observe that
\begin{align*}
\beta(r_{n+1})=\th_1\cdots\th_{2^{n+1}p}(\th_{2^{n+1}p+1}\cdots\th_{2^{n+2}p})^\f
 &=\th_1\cdots\th_{2^{n}p}\overline{\th_1\cdots\th_{2^n p}}\,^+\overline{\th_1\cdots\th_{2^n p}}\cdots\\
&<\th_1\cdots\th_{2^n p}\,(\overline{\th_1\cdots\th_{2^n p}}^+)^\f=\beta(r_n).
\end{align*}
Then by Lemma \ref{lem:21} (ii) we have $r_{n+1}<r_{n}$. Note that $\beta(q_0^*)=\al(q_0^*)=(\th_i)$, and
\[\beta(r_n)\ra (\th_i)=\beta(q_0^*)\quad\textrm{as}\quad n\ra\f.\]
Hence, we conclude from  Lemma \ref{lem:22} (iii) that $r_n\searrow q_0^*$ as $n\ra\f$. 
\end{proof}

\begin{lemma}\label{lem:32}
Let $(q_0, q_0^*)\subset (q',M+1)\setminus\overline{\U}$ be a connected component generated by  $t_1\cdots t_p$, and let  $(\th_i)=\al(q_0^*)$. Then for  any $n\ge 1$ and  any $0\le i< 2^{n}p$ we have
\begin{equation}\label{eq:32}
\begin{split}
&\overline{\th_1\cdots\th_{2^{n+1}p-i}}<\si^i(\xi_n\overline{\xi_n}) <\th_1\cdots\th_{2^{n+1}p-i},\\
&\overline{\th_1\cdots\th_{2^{n+1}p-i}}< \si^i(\xi_n\overline{\xi_n^-})  \le\th_1\cdots\th_{2^{n+1}p-i},\\
&\overline{\th_1\cdots\th_{2^{n+1}p-i}}< \si^i(\xi_n^-\xi_n) <\th_1\cdots\th_{2^{n+1}p-i},
\end{split}
\end{equation}
and thus (by symmetry),
\begin{align*}
&\overline{\th_1\cdots\th_{2^{n+1}p-i}}<\si^i(\overline{\xi_n}\xi_n) <\th_1\cdots\th_{2^{n+1}p-i},\\
&\overline{\th_1\cdots\th_{2^{n+1}p-i}}\le \si^i(\overline{\xi_n} {\xi_n^-})  <\th_1\cdots\th_{2^{n+1}p-i},\\
 &\overline{\th_1\cdots\th_{2^{n+1}p-i}}< \si^i(\overline{\xi_n^-}\,\overline{\xi_n}) <\th_1\cdots\th_{2^{n+1}p-i},
\end{align*}
 where   $\xi_n:=\th_1\cdots\th_{2^n p}$.
 \end{lemma}
\begin{proof}
By symmetry it suffices to prove (\ref{eq:32}).

Note that
$ \xi_n\overline{\xi_n}=\th_1\cdots\th_{2^{n+1}p}^-$ and $\xi_n\overline{\xi_n^-}=\th_1\cdots\th_{2^{n+1}p}.$
Then by Lemma \ref{lem:25} it follows that
\[
\overline{\th_1\cdots\th_{2^{n+1}p-i}}<\si^i(\xi_n\overline{\xi_n})<\th_1\cdots\th_{2^{n+1}p-i}
\]
and
\[
\overline{\th_1\cdots\th_{2^{n+1}p-i}}<\si^i(\xi_n\overline{\xi_n^-})\le \th_1\cdots\th_{2^{n+1}p-i}
\]
for any $0\le i<2^{n}p$.

So, it suffices to prove the inequalities
\begin{equation}\label{eq:k33}
\overline{\th_1\cdots\th_{2^{n+1}p-i}}<\si^i(\th_1\cdots\th_{2^n p}^-\th_1\cdots\th_{2^n p})<\th_1\cdots\th_{2^{n+1}p-i}
\end{equation}
for any $0\le i<2^{n}p$.
By Lemma \ref{lem:25} it follows that for any $0\le i<2^n p$ we have 
\[
\overline{\th_1\cdots\th_{2^n p-i}}\le\th_{i+1}\cdots\th_{2^n p}^-<\th_1\cdots\th_{2^n p-i}
\]
and
\[
\th_1\cdots\th_i>\overline{\th_{2^n p-i+1}\cdots\th_{2^n p}}.
\]
This proves \eqref{eq:k33}.
\end{proof}

\begin{lemma}\label{lem:33}
Let $(q_0, q_0^*)\subset (q',M+1)\setminus\overline{\U}$ be a connected component generated by  $t_1\cdots t_p$. Then
$
\dim_H\pi_{M+1}(\U_r'\setminus\U_{q_0^*}')>0
$ for any $r\in(q_0^*, M+1]$.
\end{lemma}
\begin{proof}
Take $r\in(q_0^*, M+1]$. By Lemma \ref{lem:31} there exists $n\ge 1$ such that
\[r_{n}\in(q_0^*,{  r})\cap\U.\]
Write $(\th_i)=\al(q_0^*)$  and let  $\xi_n=\th_1\cdots\th_{2^n p}$.
Denote by $X_A^{(n)}$
  the subshift of finite type over the states $\set{\xi_n, \xi_n^-, \overline{\xi_n}, \overline{\xi_n^-}}$ with adjacency matrix
  \[
A=\left(
\begin{array}{cccc}
 0&0 &1   &1   \\
  1&0&0   &0   \\
  1&1&0   &0\\
  0&0&1&0
\end{array}
\right).
\]
Note that $\al(r_n)=\th_1\cdots\th_{2^n p}(\th_{2^n p+1}\cdots\th_{2^{n+1}p})^\f$.
Then by Lemmas \ref{lem:32} and \ref{lem:23}  it follows that
\begin{equation}\label{eq:34}
X_A^{(n)}\subseteq\U_{r_{n}}'\subseteq\U_r'.
\end{equation}

 Furthermore, note that
 \begin{align*}
 \xi_n\overline{\xi_n^-}(\overline{\xi_n} \xi_n)^3&=\th_1\cdots\th_{2^{n+1}p}(\overline{\th_1\cdots\th_{2^{n+1}p}}\,^+)^3\\
&=\th_1\cdots\th_{2^{n+2}p}(\overline{\th_1\cdots\th_{2^{n+1}p}}\,^+)^2\\ 
 &>\th_1\cdots\th_{2^{n+2}p}\overline{\th_1\cdots \th_{2^{n+1}p}\th_{2^{n+1}p+1}\cdots\th_{2^{n+2}p}}\,^+\\
 &=\th_1\cdots\th_{2^{n+2}p}\th_{2^{n+2}p+1}\cdots\th_{2^{n+3}p}.
 \end{align*}
 Then by Lemmas \ref{lem:23} and \ref{lem:31} it follows that any sequence starting at
 \[\c:=\xi_n^- \xi_n\overline{\xi_n^-}(\overline{\xi_n} \xi_n)^3 \] can not belong to $\U_{r_{n+2}}'$. Therefore, by (\ref{eq:34})   we obtain
 \begin{equation}\label{eq:35}
 \begin{split}
 X_A^{(n)}(\c)  &:=\set{(d_i)\in X_A^{(n)}: d_1\cdots d_{(2^{n+3}+2^n)p}=\c} \subseteq X_A^{(n)}\setminus\U_{r_{n+2}}'  \subset\U_r'\setminus\U_{q_0^*}'.
 \end{split}
 \end{equation}

 Note that the subshift of finite type $X_A^{(n)}$ is irreducible (cf.~\cite{Lind_Marcus_1995}), and the image $\pi_{M+1}(X_A^{(n)})$ is a graph-directed set satisfying the open set condition (cf.~\cite{Mauldin_Williams_1988}). Then by (\ref{eq:35}) it follows that
 \begin{align*}
 \dim_H\pi_{M+1}(\U_r'\setminus\U_{q_0^*}') &\ge\dim_H\pi_{M+1}(X_A^{(n)}(\c))\\
 &=\dim_H\pi_{M+1}(X_A^{(n)})
 =\frac{\log \left((1+\sqrt{5})/2\right)}{2^n p\log(M+1)}>0.
 \end{align*}
\end{proof}

The following lemma can be shown in a way which resembles closely the analysis in \cite[Page 2829--2830]{Kong_Li_Dekking_2010}. For the sake of completeness  we include a sketch of its proof.
\begin{lemma}\label{lem:34}
Let $(q_0, q_0^*)\subset (q',M+1)\setminus\overline{\U}$ be a connected component. Then $\dim_H\pi_{M+1}(\U_{q_0^*}'\setminus\U_{q_0}')=0$.
\end{lemma}
 
\begin{proof}[Sketch of the proof]
Suppose that $(q_0, q_0^*)$ is a connected component generated by $\t=t_1\cdots t_p$. Then
\begin{equation}\label{eq:36}
\al(q_0)=\t^\f \quad\textrm{and}\quad \al(q_0^*)=\lim_{n\ra\f}g^n(\t)=\t^+\,\overline{\t}\;\overline{\t^+}\,\t^+ \cdots,
\end{equation}
where $g(\cdot)$ is defined in (\ref{eq:g}).

For $n\ge 0$ let $\om_n:=g^n(\t)^+$.  Take $(d_i)\in\U_{q_0^*}'\setminus\U_{q_0}'$.  Then by using (\ref{eq:36}) and Lemma \ref{lem:23}  it follows that there exists $m\ge 1$ such that
\begin{equation}\label{eq:k2}
\t^\f=\al(q_0)\le d_{m+1}d_{m+2}\cdots<\al(q_0^*)=\t^+\overline{\t} \cdots,
\end{equation}
or symmetrically,
\begin{equation}\label{eq:k3}
\t^\f=\al(q_0)\le \overline{d_{m+1}d_{m+2}\cdots}<\al(q_0^*)=\t^+\overline{\t} \cdots.
\end{equation}

 Suppose $(d_{m+i})\ne \t^\f$ and $(d_{m+i})\ne\overline{\t^\f}$. Then there exists $u\ge m$ such that
\[d_{u+1}\cdots d_{u+p}=\t^+=\om_0\quad\textrm{or}\quad d_{u+1}\cdots d_{u+p}=\overline{\t^+}=\overline{\om_0}.\]

 \begin{itemize}
   \item If $d_{u+1}\cdots d_{u+p}=\om_0=\t^+$, then by (\ref{eq:k2}) and Lemma \ref{lem:23} it follows that
       \[
       d_{u+p+1}\cdots d_{u+2p}=\overline{\t^+} \quad\textrm{or}\quad   d_{u+p+1}\cdots d_{u+2p}=\overline{\t}.
       \]
       This implies
       $d_{u+1}\cdots d_{u+2p}=\t^+\overline{\t^+} =\om_0\,\overline{\om_0}$ or $  d_{u+1}\cdots d_{u+2p}=\t^+\overline{\t}=\om_1.$

   \item If $d_{u+1}\cdots d_{u+p}=\overline{\om_0}=\overline{\t^+}$, then by (\ref{eq:k3}) and Lemma \ref{lem:23} it follows that
   \[
   d_{u+p+1}\cdots d_{u+2p}={\t^+} \quad\textrm{or}\quad d_{u+p+1}\cdots d_{u+2p}=\t.
   \]
   This yields that $d_{u+1}\cdots d_{u+2p}=\overline{\om_0}\,\om_0$ or $d_{u+1}\cdots d_{u+2p}=\overline{\om_1}$.
 \end{itemize}
 Note that for each $n\ge 0$ the word $g^n(\t)^+\,\overline{g^n(\t)}$ is a prefix of $\al(q_0^*)$.   By iteration of the above arguments, one can show that  if $d_{v+1}\cdots d_{v+2^n p}=\om_n$, then $d_{v+1}\cdots d_{v+2^{n+1}p}=\om_n\overline{\om_n}$ or $\om_{n+1}$. Symmetrically, if $d_{v+1}\cdots d_{v+2^n p}=\overline{\om_n}$, then $d_{v+1}\cdots d_{v+2^{n+1}p}=\overline{\om_n}\om_n$ or $\overline{\om_{n+1}}$.

 Hence, we conclude that $(d_{i})$ must   end with
\[
\t^* (\om_{i_0}\overline{\om_{i_0}})^{*}(\om_{i_0}\overline{\om_{j_0}})^{s_0}(\om_{i_1}\overline{\om_{i_1}})^{*}
 (\om_{i_1}\overline{\om_{j_1}})^{s_1}\cdots(\om_{i_n}\overline{\om_{i_n}})^{*}(\om_{i_n}\overline{\om_{j_n}})^{s_n}\cdots
\]
or its reflections, where   $s_n\in\set{0,1}$ and
\[
0=i_0<j_0\le i_1<j_1\le i_2<\cdots\le i_n<j_n\le i_{n+1}<\cdots.
\]
Here $*$ is an element of the set $\set{0,1,2,\ldots} \cup \set{\f}$.

 Since the length of   $\om_n=g^n(\t)^+$ grows exponentially fast as $n\ra\f$, we conclude that $\dim_H\pi_{M+1}(\U_{q_0^*}'\setminus\U_{q_0}')=0$. 
\end{proof}

\begin{proof}[Proof of Theorem \ref{th:11} for (i) $\Leftrightarrow$ (ii)]
First we prove (i) $\Rightarrow$ (ii). If $q=q_0^*$ is the right endpoint of a connected component of $(q', M+1)\setminus\overline{\U}$, then by Lemma \ref{lem:33} we have
\begin{equation*}
\dim_H\pi_{M+1}(\U_r'\setminus\U_q')>0\quad\textrm{for any }\quad r\in(q, M+1].
\end{equation*}
Clearly, it is trivial when $q=M+1$. Now we take $q\in(\U\setminus\set{M+1})\setminus\bigcup\set{q_0^*}$  and take $r\in(q, M+1]$. By Lemma \ref{lem:26} (i) one can find $q_0^*\in(q,r)$, and therefore by Lemma \ref{lem:33} we obtain
\[
\dim_H\pi_{M+1}(\U_r'\setminus\U_q')\ge\dim_H\pi_{M+1}(\U_r'\setminus\U_{q_0^*}')>0.
\]

Now we prove (ii) $\Rightarrow$ (i). Take $q\in(1, M+1]\setminus\U$. We will show that $\dim_H\pi_{M+1}(\U_r'\setminus\U_q')=0$ for some $r\in(q, M+1]$.     Note that $\bigcup\set{q_0}=\overline{\U}\setminus\U$. Then by (\ref{eq:11}) it follows that
\[
q\in(1,q')\cup\bigcup[q_0, q_0^*).
\]
Therefore, it suffices to prove $\dim_H\pi_{M+1}(\U_r'\setminus\U_q')=0$ for some $r\in(q, M+1]$. We distinct the following  two cases.

Case (I). $q\in(1,q')$. Then for any $r\in(q,q')$ we have
\[
\dim_H\pi_{M+1}(\U_r'\setminus\U_q')\le\dim_H\pi_{M+1}(\U_r')=0,
\]
where the last equality follows by \cite[Theorem 4.6]{Kong_Li_Dekking_2010} (see also, \cite{Baker_2014,Glendinning_Sidorov_2001}).

Case (II). $q\in[q_0, q_0^*)$. Then for any $r\in(q,q_0^*)$ we have by Lemma \ref{lem:34} that
\[
\dim_H\pi_{M+1}(\U_r'\setminus\U_q')\le\dim_H\pi_{M+1}(\U_{q_0^*}'\setminus\U_{q_0}')=0.
\] 
\end{proof}

\subsection{ Proof  of Theorem \ref{th:11} for (i) $\Leftrightarrow$ (iii).}
The following property  for the  Hausdorff dimension is well-known (cf.~\cite[Proposition 2.3]{Falconer_1990}).
\begin{lemma}
  \label{lem:35}
  Let $f: (X, d_1)\ra (Y, d_2)$ be a map between two metric  spaces .  If there exist constants $C>0$ and $\la>0$ such that
  \[
  d_2(f(x), f(y))\le C d_1(x, y)^\la
  \]
  for any $x, y\in X$, then $\dim_H X\ge \la\dim_H f(X)$.
\end{lemma}

\begin{lemma}\label{lem:36}
Let $q\in {\U}\setminus\set{M+1}$. Then for any $r\in(q, M+1)$ we have
\[
\dim_H\U\cap(q, r)\ge\dim_H\pi_{M+1}\left(\set{\al(p): p\in\U\cap(q, r)}\right).
\]
\end{lemma}
\begin{proof}
Fix $q\in {\U}\setminus\set{M+1}$ and $r\in(q, M+1)$. Then   Lemma \ref{lem:26}   yields that $\U\cap(q, r)$ contains infinitely many elements.
Take $p_1, p_2\in\U\cap(q, r)$ with $p_1<p_2$. Then by Lemma \ref{lem:21} we have $\al(p_1)<\al(p_2)$. So, there exists $n\ge 1$ such that
\begin{equation}\label{eq:k1}
\al_1(p_1)\cdots\al_{n-1}(p_1)=\al_1(p_2)\cdots\al_{n-1}(p_2)\quad\textrm{and}\quad \al_n(p_1)<\al_n(p_2).
\end{equation}
This implies
\begin{equation}
  \label{eq:37}
  \begin{split}
  \pi_{M+1}(\al(p_2))-\pi_{M+1}(\al(p_1))&=\sum_{i=1}^\f\frac{\al_i(p_2)-\al_i(p_1)}{(M+1)^i}\\
  &\le\sum_{i=n}^\f\frac{M}{(M+1)^i}=(M+1)^{1-n}.
  \end{split}
\end{equation}

Note that $r<M+1$. By Lemma \ref{lem:21} we have $\al(r)<\al(M+1)=M^\f$. Then there exists $N\ge 1$ such that
 \[\al_1(r)\cdots\al_N(r)<\underbrace{M\cdots M}_N.\]
   Therefore, by (\ref{eq:k1}) and Lemma \ref{lem:23} we obtain
\[
\sum_{i=1}^n\frac{\al_i(p_2)}{p_1^i}\ge\sum_{i=1}^\f\frac{\al_i(p_1)}{p_1^i}=1=\sum_{i=1}^\f\frac{\al_i(p_2)}{p_2^i}>\sum_{i=1}^n\frac{\al_i(p_2)}{p_2^i}+\frac{1}{p_2^{n+N}}.
\]
Note that $p_1, p_2$ are elements of $\U$. Then $p_2>p_1\ge q'$. This  implies
\begin{align*}
 \frac{1}{(M+1)^{n+N}}&<\frac{1}{p_2^{n+N}}<\sum_{i=1}^n\left(\frac{\al_i(p_2)}{p_1^i}-\frac{\al_i(p_2)}{p_2^i}\right)\\
 &\le\sum_{i=1}^\f\left(\frac{M}{p_1^i}-\frac{M}{p_2^i}\right)=\frac{M(p_2-p_1)}{(p_1-1)(p_2-1)}\le\frac{M(p_2-p_1)}{(q'-1)^2}.
\end{align*}
Therefore, by (\ref{eq:37}) it follows that
\[
\pi_{M+1}(\al(p_2))-\pi_{M+1}(\al(p_1))\le(M+1)^{1-n}\le\frac{(M+1)^{2+N}}{(q'-1)^2}(p_2-p_1).
\]
Furthermore, by Lemma \ref{lem:21} it follows that $\pi_{M+1}(\al(p_2))-\pi_{M+1}(\al(p_1))\ge 0$. Hence,
by using
\[f=\pi_{M+1}\circ\al:\quad \U\cap(q, r)\ra\pi_{M+1}(\set{\al(p): p\in\U\cap(q, r)})\]
 in Lemma \ref{lem:35}   we establish  the lemma. 
\end{proof}

\begin{lemma}\label{lem:37}
Let $(q_0, q_0^*)$ be a connected component of $(q', M+1)\setminus\overline{\U}$. Then $\dim_H\U\cap(q_0^*, r)>0$ for any $r\in(q_0^*, M+1]$.
\end{lemma}
\begin{proof}
Suppose that $(q_0, q_0^*)$ is a connected component generated by   $t_1\cdots t_p$.  Let $(\th_i)=\al(q_0^*)$.
For $n\ge 2$ we write $\xi_n=\th_1\cdots\th_{2^n p}$, and denote by
\[
\Gamma_n':=  \set{(d_i):  d_1\cdots d_{2^{n+1}p}=\xi_{n-1}( \overline{\xi_{n-1}}\,^+)^3, \quad (d_{2^{n+1}p+i}) \in X_A^{(n)}(\overline{\xi_n}) }.
\]
Here $X_A^{(n)}(\overline{\xi_n})$ is the follower set of $\overline{\xi_n}$ in the subshift of finite type $X_A^{(n)}$ defined in (\ref{eq:35}).
Now we claim that   any sequence $(d_i)\in\Gamma_n'$  satisfies
\begin{equation}
  \label{eq:39}
  \overline{(d_i)}<\si^j((d_i))<(d_i)\quad\textrm{for all} \quad j\ge 1.
\end{equation}

Take $(d_i)\in\Gamma_n'$. Then we deduce by the definition of $\Gamma_n'$ that
\begin{equation}
\label{eq:310}
d_1\cdots d_{2^{n+1}p+2^{n-1} p}=\th_1\cdots\th_{2^{n-1}p}(\overline{\th_1\cdots\th_{2^{n-1}p}}\,^+)^3\, \overline{\th_1\cdots\th_{2^{n}p}}.
\end{equation}
We will split the proof of (\ref{eq:39}) into the following five cases.
\begin{enumerate}
  \item[(a)] $1\le j<2^{n-1}p$.  By   (\ref{eq:310})  and Lemma \ref{lem:25} it follows that
  \[
  \overline{\th_1\cdots\th_{2^{n-1}p-j}}<d_{j+1}\cdots d_{2^{n-1}p}=\th_{j+1}\cdots\th_{2^{n-1}p}\le\th_1\cdots\th_{2^{n-1}p-j},
  \]
  and
  \[
 d_{2^{n-1}p+1}\cdots d_{2^{n-1}p+j}=\overline{\th_1\cdots\th_j}<\th_{2^{n-1}p-j+1}\cdots\th_{2^{n-1}p}.
  \]
  This implies that   (\ref{eq:39}) holds for all $1\le j<2^{n-1}p$.

  \item[(b)] $2^{n-1}p\le j<2^n p$. Let $k=j-2^{n-1}p$. Then $0\le k<2^{n-1}p$. Clearly, if $k=0$, then by using $\th_1>\overline{\th_1}$ and $n\ge 2$ it yields that
      \[
      \overline{\th_1\cdots\th_{2^{n-1}p}}<d_{j+1}\cdots d_{2^n p}=\overline{\th_1\cdots\th_{2^{n-1}p}}\,^+<\th_1\cdots\th_{2^{n-1}p}.
      \]
Now we assume $1\le k<2^{n-1}p$. Then by   (\ref{eq:310}) and Lemma \ref{lem:25} it follows that
  \[
  \overline{\th_1\cdots\th_{2^{n-1}p-k}}<d_{j+1}\cdots d_{2^n p}=\overline{\th_{k+1}\cdots\th_{2^{n-1}p}}\,^+\le\th_1\cdots\th_{2^{n-1}p-k},
  \]
  and
  \[
  d_{2^{n}p+1}\cdots d_{2^n p+k}=\overline{\th_1\cdots\th_k}<\th_{2^{n-1}p-k+1}\cdots\th_{2^{n-1}p}.
  \]
  Therefore, (\ref{eq:39}) holds for all $2^{n-1}p\le j<2^n p$.

  \item[(c)] $2^n p\le j<2^n p+2^{n-1}p$. Let $k=j-2^n p$. Then in a similar way as in Case (b) one can prove (\ref{eq:39}).

  \item [(d)] $2^np+2^{n-1}p\le j<2^{n+1}p$. Let $k=j-2^np-2^{n-1}p$. Again by  the same arguments as in Case (b) we obtain (\ref{eq:39}).

  \item [(e)] $j\ge 2^{n+1}p$. Note that
  \begin{equation*}
  d_1\cdots d_{2^{n+1}p}=\th_1\cdots\th_{2^{n-1}p}(\overline{\th_1\cdots\th_{2^{n-1}p}}\,^+)^3>\th_1\cdots\th_{2^{n+1}p}.
   \end{equation*}
   Then (\ref{eq:39}) follows by Lemma \ref{lem:32}.
\end{enumerate}

Therefore, by (\ref{eq:39}) and Lemma \ref{lem:24} it follows that any sequence in $\Gamma_n'$ corresponds to a unique base $q\in\U$. Furthermore,  by (\ref{eq:310}) and Lemma \ref{lem:31} each sequence $(d_i)\in\Gamma_n'$ satisfies
\[\al(q_0^*)=(\th_i)<(d_i)<\th_1\cdots\th_{2^{n-1}p}(\overline{\th_1\cdots\th_{2^{n-1}p}}\,^+)^\f=\al(r_{n-1}).\]
 Then by   Lemma  \ref{lem:21} it follows that \[\al(q)\in\Gamma_n'\quad\Longrightarrow\quad q\in\U\cap(q_0^*, r_{n-1}).\]
Fix $r>q_0^*$. So by Lemma \ref{lem:31} there exists a sufficiently large integer $n\ge 2$ such that
\begin{equation}\label{eq:k4}
\Gamma_n'\subset\set{\al(q): q\in\U\cap(q_0^*, r)}.
\end{equation}

Note by the proof of Lemma \ref{lem:33}   that $X_A^{(n)}$ is an irreducible subshift of finite type over the states $\set{\xi_n, \xi_n^-, \overline{\xi_n}, \overline{\xi_n^-}}$.
Hence, by (\ref{eq:k4}) and Lemma \ref{lem:36} it follows that
\begin{align*}
\dim_H\U\cap(q_0^*, r)\ge\dim_H\pi_{M+1}(\Gamma_n')&=\dim_H\pi_{M+1}(X_A^{(n)})\\
&=\frac{\log \left((1+\sqrt{5})/2\right)}{2^n p\log (M+1)}
\;>0.
\end{align*}
\end{proof}

\begin{proof}[Proof of Theorem \ref{th:11} for (i) $\Leftrightarrow$ (iii)]
First we prove (i) $\Rightarrow$ (iii).  Excluding the trivial case $q=M+1$ we take $q\in\U\setminus\set{M+1}$.  Suppose that $r \in (q,M+1]$. If $q=q_0^*$, then by Lemma \ref{lem:37} we have
$
\dim_H\U\cap(q, r)>0.
$

If $q\in (\U\setminus\set{M+1}) \setminus\bigcup\set{q_0^*}$, then by Lemma \ref{lem:26} (i) there exists $q_0^*\in(q, r)$. So, by Lemma \ref{lem:37} we have
\[
\dim_H\U\cap(q,r)\ge\dim_H\U\cap(q_0^*, r)>0.
\]

Now we prove (iii) $\Rightarrow$ (i). Suppose on the contrary that  $q\in(1,M+1]\setminus\U$. We will show that $\U\cap(q,r)=\emptyset$ for some $r\in(q, M+1]$.
Take $q\in(1, M+1]\setminus\U$. By (\ref{eq:11}) it follows that 
\[
q\in(1, q')\cap\bigcup[q_0, q_0^*).
\]
This implies  that  $\U\cap(q,r)=\emptyset$ for $r \in (q, M+1]$ sufficiently close to $q$.  
\end{proof}

\subsection{ Proof  of Theorem \ref{th:12}.}

\begin{proof}[Proof of Theorem \ref{th:12}]
(i) $\Rightarrow$ (ii). Take $q\in\overline{\U}\setminus(\bigcup\set{q_0^*}\cup\set{q'})$ and $p\in(1,q)$. By Lemma \ref{lem:26} (ii) there exists $q_0^*\in(p,q)$. Hence, by Lemma \ref{lem:33} it follows that
\[
\dim_H\pi_{M+1}(\U_q'\setminus\U_p')\ge\dim_H\pi_{M+1}(\U_q'\setminus\U_{q_0^*}')>0.
\]

(ii) $\Rightarrow$ (i).  Suppose on the contrary that $q\notin\overline{\U}\setminus(\bigcup\set{q_0^*}\cup\set{q'})$. Then by (\ref{eq:11}) we have
\[
q\in(1,q']\cup\bigcup(q_0, q_0^*].
\]
By using Lemma \ref{lem:34} it follows that for $p\in(1,q)$ sufficiently close to $q$ we have $\dim_H\pi_{M+1}(\U_q'\setminus\U_p')=0$.

(i) $\Rightarrow$ (iii).  Take $q\in\overline{\U}\setminus(\bigcup\set{q_0^*}\cup\set{q'})$ and $p\in(1,q)$. By Lemma \ref{lem:26} (ii) there exists $q_0^*\in(p,q)$. Hence, by Lemma \ref{lem:37} it follows that
\[
\dim_H\U\cap(p,q)\ge\dim_H\U\cap(q_0^*, q)>0.
\]

(iii) $\Rightarrow$ (i).  Suppose $q\notin\overline{\U}\setminus(\bigcup\set{q_0^*}\cup\set{q'})$. Then by (\ref{eq:11}) we have
$
q\in(1,q']\cup\bigcup(q_0, q_0^*].
$
So, for $p\in(1, q)$ sufficiently close to $q$ we have  $\U\cap(p,q)=\emptyset$. 
\end{proof}

\section*{Acknowledgements}
The authors thank the anonymous referees for many useful comments.  The first author was supported by   NSFC No.~11401516 and Jiangsu Province Natural Science Foundation for the Youth No. BK20130433. The second author was
supported by  NSFC No.~11271137, 11571144, 11671147 and Science and Technology Commission of Shanghai Municipality (STCSM)  No.~13dz2260400. The third author was supported by NSFC No.~11601358.

\end{document}